\documentclass[revision]{FPSAC2025}

\newcounter{thm}
\numberwithin{thm}{section}

\theoremstyle{definition}
\newtheorem{definition}[thm]{Definition}
\AddToHook{env/definition/begin}{\crefalias{thm}{definition}}
\newtheorem{example}[thm]{Example}
\AddToHook{env/example/begin}{\crefalias{thm}{example}}

\theoremstyle{plain}
\newtheorem{theorem}[thm]{Theorem}
\AddToHook{env/theorem/begin}{\crefalias{thm}{theorem}}
\newtheorem{proposition}[thm]{Proposition}
\AddToHook{env/proposition/begin}{\crefalias{thm}{proposition}}
\newtheorem{lemma}[thm]{Lemma}
\AddToHook{env/lemma/begin}{\crefalias{thm}{lemma}}
\newtheorem{corollary}[thm]{Corollary}
\AddToHook{env/corollary/begin}{\crefalias{thm}{corollary}}
\newtheorem{conjecture}[thm]{Conjecture}
\AddToHook{env/conjecture/begin}{\crefalias{thm}{conjecture}}

\theoremstyle{remark}
\newtheorem{remark}[thm]{Remark}
\AddToHook{env/remark/begin}{\crefalias{thm}{remark}}

\usepackage{commath}

\newcommand{\NN}{\mathbb{N}}
\newcommand{\ZZ}{\mathbb{Z}}
\newcommand{\QQ}{\mathbb{Q}}
\newcommand{\RR}{\mathbb{R}}
\newcommand{\CC}{\mathbb{C}}
\newcommand{\fS}{\mathfrak{S}}
\newcommand{\cC}{\mathcal{C}}
\newcommand{\cP}{\mathcal{P}}
\newcommand{\cQ}{\mathcal{Q}}
\newcommand{\sT}{\mathsf{T}}
\DeclareMathOperator{\sgn}{sgn}
\DeclareMathOperator{\Vol}{Vol}
\DeclareMathOperator{\Ind}{Ind}
\DeclareMathOperator{\Inf}{Inf}
\DeclareMathOperator{\Res}{Res}
\DeclareMathOperator{\Stab}{Stab}

\title[The number of irreducibles in the plethysm {$s_\lambda[s_m]$}]{The number of irreducibles in the plethysm \texorpdfstring{$s_\lambda[s_m]$}{s lambda[s m]}}

\author{Ming Yean Lim\thanks{\href{mailto:mylim@umich.edu}{mylim@umich.edu}.}\addressmark{1}}

\address{\addressmark{1}Department of Mathematics, University of Michigan, Ann Arbor, MI, USA}

\received{November 8, 2024}

\revised{March 26, 2025}

\abstract{We give a formula for the number of irreducibles (with multiplicity) in the decomposition of the plethysm $s_\lambda[s_m]$ of Schur functions in terms of the number of lattice points in certain rational polytopes. In the case where $\lambda = n$ consists of a single part, we will give a combinatorial interpretation of this number as the cardinality of a set of matrices modulo permutation equivalence. This is also the setting of Foulkes' conjecture, and our results allow us to state a weaker version that only involves comparing the cardinalities of such sets, rather than the multiplicities of irreducible representations.}


\usepackage{biblatex}
\begin{document}

\maketitle

\section{Introduction}\label{sec:intro}

Let $\Lambda$ denote the ring of symmetric functions over $\QQ$. \emph{Plethysm} is a binary operation $(f, g) \mapsto f[g]$ on $\Lambda$, introduced by Littlewood \cite{littlewood1936polynomial} in 1936. In modern language, it is most easily expressed in terms of the power sum symmetric functions $p_m$ as the unique operation \cite{loehr_expose} satisfying
\begin{enumerate}
    \item for $n, m \ge 1$, $p_n[p_m] = p_{nm}$;
    \item for $m \ge 1$, $g \mapsto p_m[g]$ is a $\QQ$-algebra homomorphism $\Lambda \to \Lambda$;
    \item for $g \in \Lambda$, $f \mapsto f[g]$ is a $\QQ$-algebra homomorphism $\Lambda \to \Lambda$.
\end{enumerate}

The decomposition of the plethysm of Schur functions
\begin{equation}\label{eqn:pleth_schur}
    s_\lambda[s_\mu] = \sum_{\nu \vdash nm} a_{\lambda, \mu}^\nu s_\nu
\end{equation}
for partitions $\lambda \vdash n, \mu \vdash m$ is of particular importance. Schur functions correspond to irreducible representations of symmetric groups, and from this point of view it can be shown that the \emph{plethysm coefficients} $a_{\lambda, \mu}^\nu$ are non-negative integers \cite[Sec.~5.4]{james_kerber_reps_sym}. Various formulas and algorithms have been developed to compute these coefficients \cite{colmenarejo2022mystery, kahle_plethysm, yang_algorithm}, though from a computational complexity perspective this is known to be a hard problem in general \cite{fischer2020computational}. Recent works study plethysm via the representation theory of partition algebras \cite{bowman2024partition}, party algebras \cite{orellana2022plethysm}, and geometric complexity theory \cite{dorfler2020geometric, fischer2020computational}.

We focus on the case $\mu = m$ and study instead the sum
\begin{equation}\label{eqn:pleth_sum}
    \sum_{\nu \vdash nm} a_{\lambda,m}^\nu,
\end{equation}
which is the number of irreducibles in the decomposition \eqref{eqn:pleth_schur}, counted with multiplicity.

\subsection{Notation}

Before stating our results, we recall some notions from the representation theory of finite groups and symmetric groups to fix notation.

Unless otherwise stated, $n$ and $m$ shall denote positive integers. We use $\# S$ and $\abs{S}$ to denote the cardinality of a finite set $S$. For a finite group $G$, we let
\begin{equation*}
    \langle \chi, \phi \rangle_G = \frac{1}{\abs{G}} \sum_{g \in G} \chi(g) \overline{\phi(g)}
\end{equation*}
denote the inner product of class functions $\chi$ and $\phi$ of $G$. We omit the subscript $G$ if the group is clear from context. For a subgroup $H \le G$, we write $\Ind_H^G \chi$ for the induction of a character $\chi$ from $H$ to $G$ and $\Res_H^G \chi$ for the restriction of a character $\chi$ of $G$ to $H$. For a quotient $G \twoheadrightarrow K$, denote by $\Inf_K^G \chi$ the inflation of a character $\chi$ from $K$ to $G$.

Let $G \wr \fS_n$ denote the \emph{wreath product} of $G$ with the symmetric group $\fS_n$. By definition, it is the semidirect product $G^n \rtimes \fS_n$, with $\fS_n$ acting on $G^n$ by permuting the coordinates. Following \cite[Sec.~4.1]{james_kerber_reps_sym}, we write an element of $G \wr \fS_n$ in the form $(f; \sigma)$, where $f \in G^n$ and $\sigma \in \fS_n$. For $\sigma \in \fS_n$, we shall sometimes write $\sigma$ for $(1; \sigma)$, where $1 \in G^n$ is the identity element. We now specialize to $G = \fS_m$. For $1 \le i \le n$, define
\begin{equation}\label{eqn:def_Pi}
    \cP_i = \{(i-1)m + 1, \ldots, im\}.
\end{equation}
We have inclusions $\fS_m^n \le \fS_m \wr \fS_n \le \fS_{nm}$, where $\fS_m^n$ and $\fS_m \wr \fS_n$ are identified with the stabilizers of $(\cP_1, \ldots, \cP_n)$ and $\{\cP_1, \ldots, \cP_n\}$ respectively under the natural $\fS_{nm}$-actions.

Given a partition $\lambda$, we let $\ell(\lambda)$ denote its length. We write $m^n$ for the partition of $nm$ consisting of $n$ occurrences of $m$. We generally use $\lambda \vdash n$ to index the irreducible representations of $\fS_n$ and $\rho \vdash n$ for cycle types of elements of $\fS_n$. Thus we write $\chi^\lambda(\rho)$ for the value of the irreducible character $\chi^\lambda$ of $\fS_n$ at an element of cycle type $\rho$, see \cite[Sec.~2.3]{james_kerber_reps_sym}. In particular, recall that $\chi^n = 1$ and $\chi^{1^n} = \sgn$ are the trivial and sign characters of $\fS_n$ respectively. The plethysm coefficients in our case are then given \cite[Sec.~5.4]{james_kerber_reps_sym} by
\begin{equation*}
    a_{\lambda, m}^\nu = \langle \chi^\nu, \Ind_{\fS_m \wr \fS_n}^{\fS_{nm}} \Inf_{\fS_n}^{\fS_m \wr \fS_n} \chi^\lambda \rangle.
\end{equation*}

\subsection{Statement of results}

We now state our main results. Let $M(n, m)$ denote the set of $n \times n$ matrices with non-negative integer entries whose row and column sums are all equal to $m$. We shall identify $\fS_n$ with the group of $n \times n$ permutation matrices.

\begin{definition}
    Define the function $N^m : \fS_n \to \ZZ$ by
    \begin{equation*}
        N^m(\sigma) = \# \{A \in M(n, m) \mid \sigma A^\sT = A \}
    \end{equation*}
    for $\sigma \in \fS_n$, where $A^\sT$ denotes the transpose of $A$.
\end{definition}

It is easy to see that $N^m$ is a class function of $\fS_n$. Without the transpose, it would be the permutation character of $M(n, m)$. The main result is as follows:

\begin{theorem}\label{thm:main_thm}
    For integers $n, m \ge 1$, $N^m$ is a character of $\fS_n$. Moreover for $\lambda \vdash n$,
    \begin{equation*}
        \langle \chi^\lambda, N^m \rangle = \sum_{\nu \vdash nm} a_{\lambda, m}^\nu.
    \end{equation*}
\end{theorem}
We shall prove this result in the next section. As we shall explain further in \cref{sec:properties}, for fixed $\sigma \in \fS_n$, $N^m(\sigma)$ is an Ehrhart quasipolynomial in $m$. Thus for fixed $\lambda \vdash n$, the sum \eqref{eqn:pleth_sum} is quasipolynomial in $m$. Related to this, according to \cite{kahle2018obstructions}, the function $s \mapsto a_{n, sm}^{s \nu}$ is a quasipolynomial by a deep result, but is not an Ehrhart quasipolynomial. In the case $n=3$, these quasipolynomials were computed in \cite{agaoka_decomposition} using the decomposition of $s_\lambda[s_m]$ for $\lambda \vdash 3$. With the aid of a computer, \cref{thm:main_thm} enables us to compute the quasipolynomials \eqref{eqn:pleth_sum} for all $\lambda \vdash n$ with $n \le 6$.

\begin{example}
    For $n=6$ and $\lambda = 6$, we have computed using SageMath \cite{sagemath} that
    \begin{align*}
        \sum_{\nu \vdash 6m} a_{6, m}^\nu ={}& \frac{243653}{1434705592320000} m^{15} + \frac{243653}{31882346496000} m^{14} + \frac{91173671}{573882236928000} m^{13} \\
        \phantom{\sum_{\nu \vdash 6m} a_{6, m}^\nu} &+ \frac{5954623}{2942985830400} m^{12} + \frac{3895930519}{220723937280000} m^{11} + \frac{149644967}{1337720832000} m^{10} \\
        &+ \frac{1072677673}{2006581248000} m^9 + \frac{14723521}{7431782400} m^8 + \frac{350041981}{59719680000} m^7 + O(m^6),
    \end{align*}
    where we have omitted the trailing terms whose coefficients have period greater than $1$.
\end{example}

Asymptotics of \eqref{eqn:pleth_sum} as $m \to \infty$ were studied in \cite{fulger_asymptotics}. We recover and slightly extend one of their results in \cref{thm:asymptotics} by studying the dimensions of the polytopes involved.

We call two matrices $A, B \in M(n, m)$ \emph{permutation equivalent} and write $A \sim B$ if $A$ can be transformed into $B$ by row and column permutations. Let $T(n, m) = \{A \in M(n, m) \mid A \sim A^\sT\}$ denote the subset of matrices that are permutation equivalent to their transpose. A major open problem in algebraic combinatorics is to find a combinatorial interpretation of plethysm coefficients $a_{\lambda, \mu}^\nu$ \cite[Problem 9]{stanley_positivity}. We have the following combinatorial interpretation of the sum \eqref{eqn:pleth_sum} in the case $\lambda = n$:

\begin{theorem}\label{thm:combinatorial_interpretation}
    When $\lambda = n$, the sum \eqref{eqn:pleth_sum} is equal to the cardinality of $T(n, m) / {\sim}$:
    \begin{equation*}
        \sum_{\nu \vdash nm} a_{n,m}^\nu = \langle 1, N^m \rangle_{\fS_n} = \# T(n, m) / {\sim}.
    \end{equation*}
\end{theorem}

\begin{example}
    We compute that $s_3[s_3] = s_{9} + s_{72} + s_{63} + s_{522} + s_{441}$. Correspondingly, there are five elements of $T(3,3) / {\sim}$, represented by
    \begin{equation*}
        \begin{pmatrix}
            3 & 0 & 0 \\
            0 & 3 & 0 \\
            0 & 0 & 3
        \end{pmatrix},
        \begin{pmatrix}
            3 & 0 & 0 \\
            0 & 2 & 1 \\
            0 & 1 & 2
        \end{pmatrix},
        \begin{pmatrix}
            2 & 1 & 0 \\
            1 & 1 & 1 \\
            0 & 1 & 2
        \end{pmatrix},
        \begin{pmatrix}
            2 & 1 & 0 \\
            1 & 0 & 2 \\
            0 & 2 & 1
        \end{pmatrix},
        \begin{pmatrix}
            1 & 1 & 1 \\
            1 & 1 & 1 \\
            1 & 1 & 1
        \end{pmatrix}.
    \end{equation*}
\end{example}

In the above setting, Foulkes \cite{foulkes_concomitants} conjectured in 1950 that if $n \le m$, then $a_{n,m}^\nu \le a_{m,n}^\nu$. Brion \cite{brion1993stable} proved this for $n \ll m$ in 1993. Among recent works on Foulkes' conjecture, \cite{dorfler2020geometric} verifies it for an infinite family, and a stable version was proven in \cite{bowman2024partition} via partition algebras. Together with the previous theorem, Foulkes' conjecture implies

\begin{conjecture}\label{conj:T_ineq}
    If $n \le m$, then $\# T(n, m) / {\sim} \le \# T(m, n) / {\sim}$.
\end{conjecture}

\section{Proof of \texorpdfstring{\cref{thm:main_thm}}{Theorem 1.2}}\label{sec:main_prf}

In this section we prove \cref{thm:main_thm}. Define the function $\theta : \fS_{nm} \to \CC$ by
\begin{equation*}
    \theta(\sigma) = \# \{\tau \in \fS_{nm} \mid \tau^2 = \sigma\}.
\end{equation*}
It is well-known that the irreducible representations of $\fS_{nm}$ can be realized over $\RR$, see e.g.~\cite[Theorem 2.1.12]{james_kerber_reps_sym}. Hence \cite[Corollary 23.17]{james_liebeck_reps_groups} implies
\begin{equation*}
    \theta = \sum_{\nu \vdash nm} \chi^\nu,
\end{equation*}
and is in particular a character. We compute by Frobenius reciprocity
\begin{align*}
    \sum_{\nu \vdash nm} a_{\lambda, m}^\nu
    &= \langle \theta, \Ind_{\fS_m \wr \fS_n}^{\fS_{nm}} \Inf_{\fS_n}^{\fS_m \wr \fS_n} \chi^\lambda \rangle
    = \langle \Res_{\fS_m \wr \fS_n}^{\fS_{nm}} \theta, \Inf_{\fS_n}^{\fS_m \wr \fS_n} \chi^\lambda \rangle \\
    &= \frac{1}{m!^n n!}  \sum_{(f; \sigma) \in \fS_m \wr \fS_n} \theta(f; \sigma) \chi^{\lambda}(\sigma)
    = \frac{1}{n!} \sum_{\sigma \in \fS_n} \left(\frac{1}{m!^n} \sum_{f \in \fS_m^n} \theta(f; \sigma) \right) \chi^{\lambda}(\sigma) .
\end{align*}
The expression in parentheses simplifies as
\begin{equation*}
    \frac{1}{m!^n} \sum_{f \in \fS_m^n} \theta(f; \sigma)
    = \frac{1}{m!^n} \sum_{f \in \fS_m^n} \# \{\tau \in \fS_{nm} \mid \tau^2 = (f; \sigma)\}
    = \frac{1}{m!^n} \# \{\tau \in \fS_{nm} \mid \tau^2 \sigma^{-1} \in \fS_m^n \} .
\end{equation*}
Thus it suffices to show that for $\sigma \in \fS_n$,
\begin{equation*}
    \frac{1}{m!^n} \# \{\tau \in \fS_{nm} \mid \tau^2 \sigma^{-1} \in \fS_m^n \} = N^m(\sigma) .
\end{equation*}
Now fix $\sigma \in \fS_n$ and recall the definition of $\cP_i$ in \eqref{eqn:def_Pi}. The proof is completed by
\begin{lemma}
    The map $F : \{\tau \in \fS_{nm} \mid \tau^2 \sigma^{-1} \in \fS_m^n \} \to \{A \in M(n, m) \mid \sigma A^\sT = A \}$ given by
\begin{equation*}
    F(\tau)_{ij} = \# (\cP_i \cap \tau \cP_j)
\end{equation*}
is $m!^n$-to-$1$ and surjective.
\end{lemma}
\begin{proof}
    We first show that $F$ has the stated codomain. Let $\tau \in \fS_{nm}$ be such that $\tau^2 \sigma^{-1} \in \fS_m^n$, or equivalently $\tau^2 \cP_i = \cP_{\sigma(i)}$ for each $1 \le i \le n$. Let $A = F(\tau)$ and for $1 \le i, j \le n$, set $\cQ_{ij} = \cP_i \cap \tau \cP_j$. Then
    \begin{enumerate}[label=(\roman*)]
        \item\label{item:part} for each $1 \le i \le n$, $\{\cQ_{ij} \mid 1 \le j \le n\}$ is a partition of $\cP_i$ with $\# \cQ_{ij} = A_{ij}$; and
        \item\label{item:bij} for each $1 \le i, j \le n$, $\tau$ restricts to a bijection $\cQ_{ij} \to \cQ_{\sigma(j)i}$.
    \end{enumerate}
    Indeed \ref{item:part} is clear, and \ref{item:bij} follows from $\tau \cQ_{ij} = \tau \cP_i \cap \tau^2 \cP_j = \tau \cP_i \cap \cP_{\sigma(j)} = \cQ_{\sigma(j)i}$. It follows from \ref{item:part} that $A$ has row sums equal to $m$. The condition $\sigma A^\sT = A$ is equivalent to $A_{ij} = A_{\sigma(j) i}$ for $1 \le i, j \le n$, which holds by \ref{item:bij}.

    Conversely let $A \in M(n, m)$ be such that $\sigma A^\sT = A$. We want to count $\tau \in \fS_{nm}$ such that $\tau^2 \sigma^{-1} \in \fS_m^n$ and $F(\tau) = A$. Reversing the above, giving such $\tau$ is equivalent to giving $\cQ_{ij}$ for $1 \le i, j \le n$ satisfying \ref{item:part}, and bijections $\cQ_{ij} \to \cQ_{\sigma(j) i}$ as in \ref{item:bij}. There are
    \begin{equation*}
        \binom{m}{A_{i1}, \ldots, A_{in}}
    \end{equation*}
    partitions of $\cP_i$ satisfying \ref{item:part}, and $A_{ij}!$ bijections in \ref{item:bij} since $A_{ij} = A_{\sigma(j) i}$. Thus there are
    \begin{equation*}
        \prod_{i=1}^n \binom{m}{A_{i1}, \ldots, A_{in}} \prod_{1 \le i,j \le n} A_{ij}! = m!^n
    \end{equation*}
    such $\tau$ as required.
\end{proof}

\section{\texorpdfstring{$(G \wr C_2)$}{(G wr C2)}-sets and representations}\label{sec:wr_sets}

Let $G$ be a finite group and let $C_2$ denote the cyclic group of order $2$, whose generator we shall suggestively denote by $\sT$. Elements of $G \wr C_2$ have the form $(g, h; 1)$ or $(g, h; \sT)$ for $g, h \in G$. In this section we collect some results on $(G \wr C_2)$-actions and representations. We will obtain \cref{thm:combinatorial_interpretation} as a consequence of more general results.

\subsection{Irreducible characters of \texorpdfstring{$G \wr C_2$}{G wr C2}}

We describe the irreducible characters of $G \wr C_2$ in terms of the irreducible characters of $G$. For characters $\chi_1$ and $\chi_2$ of $G$, let $\chi_1 \boxtimes \chi_2$ denote their external tensor product. It is a character of $G \times G$ with
\begin{equation*}
    (\chi_1 \boxtimes \chi_2)(g, h) = \chi_1(g) \chi_2(h).
\end{equation*}
Suppose $V$ is a representation of $G$. Then $V \otimes V$ is a representation of $G \wr C_2$ with action
\begin{equation*}
    (g, h; 1) (v \otimes w) = gv \otimes hw; \qquad
    (1, 1; \sT) (v \otimes w) = w \otimes v.
\end{equation*}
If $\chi$ is the character of $V$, then the character $\widetilde{\chi}$ of $V \otimes V$ is given by
\begin{equation*}
    \widetilde{\chi}(g, h; 1) = \chi(g) \chi(h); \qquad
    \widetilde{\chi}(g, h; \sT) = \chi(gh).
\end{equation*}
Thus given a character $\chi$ of $G$, we can define the character $\widetilde{\chi}$ of $G \wr C_2$ by the above. By \cite[Theorem 4.4.3]{james_kerber_reps_sym}, there are three types of irreducible characters of $G \wr C_2$, these are:
\begin{enumerate}
    \item $\widetilde{\chi}$, where $\chi$ is an irreducible character of $G$;

    \item $\widetilde{\chi} \cdot \Inf_{C_2}^{G \wr C_2} \sgn$, where $\chi$ is an irreducible character of $G$;

    \item $\Ind_{G \times G}^{G \wr C_2} (\chi_1 \boxtimes \chi_2)$, where $\chi_1, \chi_2$ are distinct irreducible characters of $G$.
\end{enumerate}
By a case analysis, we obtain
\begin{proposition}\label{prop:twisted_virtual_char}
    Let $\psi$ be a character of $G \wr C_2$. Then $\psi^{\sT}(g) = \psi(g, 1; \sT)$ is a virtual character of $G$. Moreover if $\chi$ is an irreducible character of $G$, then
    \begin{equation*}
        \langle \chi, \psi^\sT \rangle = \langle \widetilde{\chi}, \psi \rangle - \langle \widetilde{\chi} \cdot \Inf_{C_2}^{G \wr C_2} \sgn, \psi \rangle.
    \end{equation*}
\end{proposition}

\subsection{Orbit-counting with a twist}

Let $X$ be a finite $(G \wr C_2)$-set. For $x \in X$ and $g, h \in G$, we write $g x = (g, 1; 1) x$, $x h = (1, h^{-1}; 1) x$, and $x^\sT = (1, 1; \sT) x$. In this way, $X$ is equipped with left and right $G$-actions and an involution compatible with each other, i.e.
\begin{enumerate}
    \item $g(xh) = (gx)h$, which we simply denote by $gxh$; and
    \item $(gxh)^\sT = h^{-1} x^\sT g^{-1}$.
\end{enumerate}
Conversely given $X$ with left and right $G$-actions and an involution $x \mapsto x^\sT$ satisfying the above, we have a $(G \wr C_2)$-action on $X$.

Let $(G \times G) \backslash X$ denote the set of $(G \times G)$-orbits of $X$. For $x, y \in X$, write $x \sim y$ if they lie in the same $(G \times G)$-orbit, i.e.~if $(G \times G) \cdot x = (G \times G) \cdot y$. The set $(G \times G) \backslash X$ inherits a $C_2$-action, given by
\begin{equation*}
    ((G \times G) \cdot x)^{\sT} = (G \times G) \cdot x^\sT .
\end{equation*}
We let $((G \times G) \backslash X)^{C_2}$ denote the subset of fixed points under the $C_2$-action. This is a set we are interested in counting. Our main example is the following:

\begin{example}\label{ex:GX}
    Let $G = \fS_n$ and $X = M(n, m)$. We have a natural $(\fS_n \wr C_2)$-action on $X$, where $\fS_n$ acts by permuting rows on the left and permuting columns on the right, while $\sT$ acts by taking transposes. Two matrices $A, B \in M(n, m)$ are permutation equivalent precisely when they lie in the same $(\fS_n \times \fS_n)$-orbit. We have $((\fS_n \times \fS_n) \backslash X)^{C_2} = T(n, m) / {\sim}$.
\end{example}

For an arbitrary finite $(G \wr C_2)$-set $X$ as before, define
\begin{equation*}
    N(g) = \# \{x \in X \mid g x^\sT = x\} .
\end{equation*}
For $s \in X$, let $\Stab_{G \times G}(s)$ denote the stabilizer of $s$ under the $(G \times G)$-action and define
\begin{equation*}
    N^s(g) = \# \{x \in (G \times G) \cdot s \mid g x^\sT = x \} .
\end{equation*}
It is easy to see that $N$ and $N^s$ are class functions of $G$, and $N^s$ only depends on the $(G \times G)$-orbit of $s$. We can say more:

\begin{proposition}\label{prop:Ns_inner_prod}
    For $s \in X$, $N^s$ is a virtual character of $G$. Moreover if $\chi$ is an irreducible character of $G$, then
    \begin{equation*}
        \langle \chi, N^s \rangle = \frac{1}{\abs{\Stab_{G \times G}(s)}} \sum_{\substack{g, h \in G \\ g s^\sT h^{-1} = s}} \chi(gh) .
    \end{equation*}
\end{proposition}

Applying this to the trivial character of $G$, we obtain
\begin{corollary}\label{cor:N_inner_product_triv}
    For $s \in X$, we have
    \begin{equation*}
        \langle 1, N^s \rangle =
        \begin{cases}
            1 & \text{if } s \sim s^\sT \\
            0 & \text{otherwise} .
        \end{cases}
    \end{equation*}
\end{corollary}

This gives the following generalization of the Cauchy--Frobenius lemma:
\begin{theorem}\label{thm:not_burnside_analog}
    $N$ is a virtual character of $G$. Furthermore, we have
    \begin{equation*}
        \# ((G \times G) \backslash X)^{C_2} = \langle 1, N \rangle = \frac{1}{\abs{G}} \sum_{g \in G} N(g) .
    \end{equation*}
\end{theorem}

\begin{proof}[Proof of \cref{thm:combinatorial_interpretation}]
    The first equality of \cref{thm:combinatorial_interpretation} is a special case of \cref{thm:main_thm} with $\lambda = n$. The second equality follows immediately from \cref{thm:not_burnside_analog} with $G$ and $X$ as in \cref{ex:GX}.
\end{proof}

\begin{remark}
    The equality $\# T(n, m) / {\sim} = \sum_{\nu \vdash nm} a_{n,m}^\nu$ was the original motivation for this paper. OEIS sequence \href{https://oeis.org/A333737}{A333737} gives the number of non-negative integer symmetric matrices with equal row sums, up to permutation equivalence. This is related to, but not exactly the sequence we have, as $n \times n$ matrices that are permutation equivalent to their transpose need not be permutation equivalent to a symmetric matrix when $n \ge 6$. See \cite[Example 1]{craigen_transpose} for a counterexample.
\end{remark}

\begin{remark}
    A related quantity to \eqref{eqn:pleth_sum} is the number of irreducibles in the decomposition of $(s_m)^n$. This is equal to $N^m(1_{\fS_n})$, i.e.~the number of $n \times n$ non-negative integer symmetric matrices with row sums equal to $m$. This follows from \cref{thm:main_thm} and \cite[Corollary 7.12.5]{stanley_ec2}. Alternatively, a bijective proof can be given by using Young's rule (take $\mu = m^n$ in \cite[Corollary 7.12.4]{stanley_ec2}) and the RSK algorithm \cite[Theorems 7.11.5, 7.13.1]{stanley_ec2}.
\end{remark}

\section{Properties of \texorpdfstring{$N^m$}{Nm}}\label{sec:properties}

Now we shall describe some properties of the characters $N^m$. For $\sigma \in \fS_n$, we write $N_\sigma(m)$ for $N^m(\sigma)$ when we want to emphasize that it is a function of $m$.  If $\sigma$ is of cycle type $\rho \vdash n$, we also set $N^m(\rho) = N_\rho(m) = N^m(\sigma)$. Thus, \cref{thm:main_thm} asserts that
\begin{equation*}
    \sum_{\nu \vdash nm} a_{\lambda, m}^\nu = \sum_{\rho \vdash n} z_\rho^{-1} \chi^\lambda(\rho) N^m(\rho) ,
\end{equation*}
where $z_\rho = \prod_{i \ge 1} i^{m_i} m_i!$, with $m_i$ being the number of parts of $\rho$ equal to $i$.

\subsection{Lattice points in polytopes}

We study the $N_\sigma$ as defined above from the point of view of Ehrhart theory, see \cite{beck_robins_polyhedra} or \cite[Ch.~4]{stanley_ec1} for an introduction to this subject.

Recall that a \emph{rational convex polytope} $\cP$ is the convex hull of finitely many points with rational coordinates in some Euclidean space $\RR^n$. Given such a polytope $\cP$, let $L_\cP(t) = \#(t \cP \cap \ZZ^n)$ and let $\Vol(\cP)$ denote its relative volume. A \emph{quasipolynomial of degree $n$} is a function $f : \ZZ \to \CC$ of the form $f(t) = c_n(t) t^n + \ldots + c_1(t) t + c_0(t)$ where $c_i : \ZZ \to \CC$ are periodic functions with $c_n \neq 0$. By Ehrhart's theorem \cite[Theorem 3.23]{beck_robins_polyhedra}, $L_\cP$ is a quasipolynomial, known as the \emph{Ehrhart quasipolynomial} of $\cP$, of degree equal to the dimension of $\cP$.

We now apply this theory in our context. For a set $S$, let $M_n(S)$ denote the set of $n \times n$ matrices with entries in $S$. For $\sigma \in \fS_n$, define the rational convex polytope
\begin{equation*}
    \cP(\sigma) = \{A \in M_n(\RR_{\ge 0}) \mid A \text{ has row sums equal to } 1 \text{ and } \sigma A^\sT = A \} \subseteq M_n(\RR) .
\end{equation*}
Then for $\sigma \in \fS_n$, $N_\sigma(m) = L_{\cP(\sigma)}(m)$, so $N_\sigma$ is a quasipolynomial of degree equal to the dimension of $\cP(\sigma)$. We have

\begin{proposition}
    For $\rho \vdash n$, the degree of $N_\rho$ is
    \begin{equation*}
        \deg N_\rho = \sum_{1 \le i < j \le \ell(\rho)} \gcd(\rho_i, \rho_j) + \sum_{1 \le i \le \ell(\rho)} \left\lfloor \frac{\rho_i - 1}{2} \right\rfloor .
    \end{equation*}
\end{proposition}

\begin{corollary}\label{cor:deg_N_parity}
    For $\rho \vdash n$, we have $(-1)^{\deg N_\rho} = (-1)^{\frac{n(n-1)}{2}} \varepsilon_\rho^{n-1}$, where $\varepsilon_\rho = (-1)^{n - \ell(\rho)}$.
\end{corollary}

\begin{corollary}
    For $\rho \vdash n$, we have $\deg N_\rho \le \frac{n(n-1)}{2}$ with equality if and only if $\rho = 1^n$.
\end{corollary}

We deduce from this and a similar analysis of the second-highest degree term the following asymptotics for \eqref{eqn:pleth_sum}, c.f.~\cite[Theorem 3.7 (i)]{fulger_asymptotics}:
\begin{theorem}\label{thm:asymptotics}
    For $\lambda \vdash n$, as $m \to \infty$, we have
    \begin{enumerate}[label=(\roman*)]
        \item $\displaystyle \sum_{\nu \vdash nm} a_{\lambda,m}^\nu \sim \frac{\chi^\lambda(1)}{n!} \Vol(\cP(1_{\fS_n})) m^{\frac{n(n-1)}{2}}$; and
        \item $\displaystyle \sum_{\nu \vdash nm} a_{\lambda,m}^\nu = \chi^\lambda(1) \sum_{\nu \vdash nm} a_{n,m}^\nu + O\left(m^{\frac{(n-1)(n-2)}{2}}\right)$.
    \end{enumerate}
\end{theorem}

The fact that $N_\sigma$ is a quasipolynomial allows us to extend its definition to all integers, and hence define a class function $N^m$ for all $m \in \ZZ$.
\begin{proposition}\label{prop:N_neg}
    For $n \ge 1$ and $m \in \ZZ$, we have
    \begin{equation*}
        N^m = (-1)^{\frac{n(n-1)}{2}} \sgn^{n-1} \cdot N^{-m - n}
    \end{equation*}
    as virtual characters of $\fS_n$. Furthermore, $N^m = 0$ for $-n+1 \le m \le -1$.
\end{proposition}
\begin{proof}[Proof Sketch]
    This follows by Ehrhart--Macdonald reciprocity \cite[Theorem 4.1]{beck_robins_polyhedra} and \cref{cor:deg_N_parity}.
\end{proof}

\subsection{A decomposition of \texorpdfstring{$N^m$}{Nm}}

Let $n, m \ge 1$ be integers. As in \cref{sec:wr_sets}, there is a natural way to decompose the character $N^m$ of $\fS_n$ as a sum
\begin{equation*}
    N^m = \sum_{\cC \in T(n, m) / {\sim}} N^\cC ,
\end{equation*}
where for $\cC \in T(n, m) / {\sim}$, we set $N^\cC(\sigma) = \# \{A \in \cC \mid \sigma A^\sT = A \}$.
If $S \in \cC$, then in the notation of \cref{sec:wr_sets}, $N^\cC = N^S$ and it is a virtual character of $\fS_n$ by virtue of \cref{prop:twisted_virtual_char}. \cref{thm:main_thm} implies that for $\lambda \vdash n$,
\begin{equation*}\label{eqn:N_decomp}
    \sum_{\nu \vdash nm} a_{\lambda, m}^\nu = \sum_{\cC \in T(n, m) / {\sim}} \langle \chi^\lambda, N^\cC \rangle .
\end{equation*}
This allows us to measure the contribution of each $\cC \in T(n, m) / {\sim}$ to the sum \eqref{eqn:pleth_sum}.
\begin{proposition}\label{prop:NC_triv}
    For $\cC \in T(n, m) / {\sim}$, we have $\langle 1, N^\cC \rangle = 1$.
\end{proposition}

\begin{remark}
    If $\cC_1 \in T(n, m_1) / {\sim}$ and $\cC_2 \in T(n, m_2) / {\sim}$ are such that there are $A_1 \in \cC_1, A_2 \in \cC_2$, and a bijection $f : \NN \to \NN$ such that $f(A_1) = A_2$, then $N^{\cC_1} = N^{\cC_2}$.
\end{remark}

\begin{remark}\label{rem:NC_virtual}
    Unfortunately $N^\cC$ is not a character in general, for example for the matrix in \cite[Example 1]{craigen_transpose}. By \cref{prop:Ns_inner_prod}, we see that $-\chi^\lambda(1) \le \langle \chi^\lambda, N^\cC \rangle \le \chi^\lambda(1)$. By \cref{thm:combinatorial_interpretation} and \cref{thm:asymptotics},
    \begin{equation*}
        \lim_{m \to \infty} \frac{1}{\abs{T(n,m) / {\sim}}} \sum_{\cC \in T(n, m) / {\sim}} \langle \chi^\lambda, N^\cC \rangle = \chi^\lambda(1) .
    \end{equation*}
    Thus when $m$ is large, for most $\cC \in T(n, m) / {\sim}$, $N^\cC$ will be the character of the regular representation of $\fS_n$.
\end{remark}

\section{The case \texorpdfstring{$m=2$}{m=2}}\label{sec:m2}

We conclude by using our results to study the case $m=2$ .

Let $n_1, n_2 \ge 1$. For $\cC_1 = [A_1] \in T(n_1, 2) / {\sim}$ and $\cC_2 = [A_2] \in T(n_2, 2) / {\sim}$, we define their sum $\cC_1 + \cC_2 \in T(n_1 + n_2, 2) / {\sim}$ to be the equivalence class of the block diagonal matrix $A_1 \oplus A_2$. We call $\cC \in T(n, 2) / {\sim}$ \emph{irreducible} if it cannot be written as a nontrivial sum $\cC_1 + \cC_2$. They are represented by matrices of the form

\begin{equation*}
    \begin{pmatrix}
        2
    \end{pmatrix},
    \begin{pmatrix}
        1 & 1\\
        1 & 1
    \end{pmatrix},
    \begin{pmatrix}
        1 & 1 &  \\
        1 &   & 1\\
          & 1 & 1
    \end{pmatrix},
    \begin{pmatrix}
        1 & 1 &   &  \\
        1 &   & 1 &  \\
          & 1 &   & 1\\
          &   & 1 & 1
    \end{pmatrix},
    \begin{pmatrix}
        1 & 1 &   &   &  \\
        1 &   & 1 &   &  \\
          & 1 &   & 1 &  \\
          &   & 1 &   & 1\\
          &   &   & 1 & 1\\
    \end{pmatrix}, \ldots
\end{equation*}
Furthermore, each $\cC \in T(n, 2) / {\sim}$ can be expressed as a sum of irreducibles, unique up to ordering. We thus have a bijection between $T(n, 2) / {\sim}$ and the set of partitions of $n$, sending an equivalence class $\cC$ to the partition $\lambda_\cC \vdash n$ recording the sizes of the irreducible summands.

\begin{example}
    If $\cC \in T(5,2) / {\sim}$ is the equivalence class of
    \begin{equation*}
        \begin{pmatrix}
            1 & 1 &   &   &   \\
            1 &   & 1 &   &   \\
              & 1 & 1 &   &   \\
              &   &   & 2 &   \\
              &   &   &   & 2
        \end{pmatrix},
    \end{equation*}
    then $\lambda_\cC = (3, 1, 1)$.
\end{example}

\begin{proposition}\label{prop:NC_inner_prod_2}
    For $\cC \in T(n, 2) / {\sim}$, we have
    \begin{equation*}
        \langle \sgn, N^\cC \rangle_{\fS_n} =
        \begin{cases}
            1 & \text{if all parts of } \lambda_\cC \text{ are odd} \\
            0 & \text{otherwise} .
        \end{cases}
    \end{equation*}
\end{proposition}

\begin{corollary}\label{cor:N2}
    We have
    \begin{enumerate}[label=(\roman*)]
        \item $\displaystyle \langle 1, N^2 \rangle_{\fS_n} = \sum_{\nu \vdash 2n} a_{n, 2}^\nu = \#\{\text{partitions of } n\}$; and
        \item $\displaystyle \langle \sgn, N^2 \rangle_{\fS_n} = \sum_{\nu \vdash 2n} a_{1^n, 2}^\nu = \#\{\text{partitions of } n \text{ into odd parts}\}$.
    \end{enumerate}
\end{corollary}

While the decompositions of $s_n[s_2]$ and $s_{1^n}[s_2]$ are well-known \cite[p.~138]{macdonald_symm_fn}, the above corollary was derived independently of these. As in \cref{rem:NC_virtual}, for general $m$, it is possible but rare that $\langle \sgn, N^\cC \rangle = -1$. We hope that a better understanding of this will lead to a combinatorial interpretation of \eqref{eqn:pleth_sum} for $\lambda = 1^n$.

\acknowledgements{%
The author would like to thank Thomas Lam for introducing them to subject of symmetric functions in a course taught at the University of Michigan during the Winter term of 2024, and for helpful discussions in preparing this paper. In this process, we have used the OEIS \cite{oeis} database, and computer algebra systems GAP \cite{GAP4} and SageMath \cite{sagemath} to perform computations.
}

\printbibliography

\end{document}